\documentclass[a4paper,11pt,twoside,english]{article}
\usepackage{babel}
\usepackage{latexsym,amsfonts}
\usepackage{amsthm}
\usepackage{amsmath}
\usepackage{paralist}
\usepackage{mathtools}

\title{Some remarks on stronger versions of the \\ Boundary Problem for Banach spaces}
\author{Jan-David Hardtke}
\date{}

\DeclareMathOperator{\co}{co}
\DeclareMathOperator{\ex}{ex}
\DeclareMathOperator{\lin}{span}

\providecommand{\ncl}[1]{\overline{#1}}
\providecommand{\cl}[2]{\overline{#1}^{#2}}
\providecommand{\ncco}[1]{\overline{\co} #1}
\providecommand{\cco}[2]{\overline{\co}^{#2} #1}
\providecommand{\dist}[2]{\operatorname{dist}(#1,#2)}
\providecommand{\keywords}[1]{\par \noindent {\scriptsize {\bf Keywords:} #1}}
\providecommand{\AMS}[1]{\par \noindent {\scriptsize {\bf AMS Subject Classification (2000):} #1}}

\DeclarePairedDelimiter{\abs}{\lvert}{\rvert}
\DeclarePairedDelimiter{\set}{\lbrace}{\rbrace}
\DeclarePairedDelimiter{\norm}{\lVert}{\rVert}
\DeclarePairedDelimiter{\paren}{\lparen}{\rparen}

\newtheorem{theorem}{Theorem}[section]
\newtheorem{corollary}[theorem]{Corollary}
\newtheorem{lemma}[theorem]{Lemma}
\newtheorem{proposition}[theorem]{Proposition}

\usepackage{color}
\definecolor{darkgreen}{rgb}{0,0.5,0}
\usepackage[colorlinks,linkcolor=blue,citecolor=red,urlcolor=darkgreen]{hyperref}

\begin{document}

\maketitle

\begin{abstract}
\noindent Let $X$ be a real Banach space. A subset $B$ of the dual unit sphere of $X$ is said to be a boundary for $X$, if every element of $X$ attains its norm on some functional in $B$. The well-known Boundary Problem originally posed by Godefroy asks whether a bounded subset of $X$ which is compact in the topology of pointwise convergence on $B$ is already weakly compact. This problem was recently solved by  Pfitzner in the positive. In this note we collect some stronger versions of the solution to the Boundary Problem, most of which are restricted to special types of Banach spaces. We shall use the results and techniques of Pfitzner, Cascales et al., Moors and others.
\end{abstract}

\AMS{46A50; 46B50}

\section{Introduction}

First we fix some notation: Throughout this paper $X$ denotes a real Banach space, $X^{*}$ its dual, $B_X$ its closed unit ball and $S_X$ its unit sphere. For a subset $B$ of $X^{*}$ we denote by $\sigma_B$ the topology on $X$ of pointwise convergence on $B$. If $A\subseteq X$, then $\co A$ stands for the convex hull of $A$ and $\cl{A}{\tau}$ for the closure of $A$ in any topology $\tau$ on $X$, except for the norm closure, which we simply denote by $\ncl{A}$. Also, we denote by $\ex C$ the set of extreme points of a convex subset $C$ of $X$. \par

Now recall that a subset $B$ of $S_{X^{*}}$ is called a boundary for $X$, if for every $x\in X$ there is some $b\in B$ such that $b(x)=\norm*{x}$. It easily follows from the Krein-Milman theorem that $\ex B_{X^{*}}$ is always a boundary for $X$. In 1980 Bourgain and Talagrand proved in \cite{bourgain} that a bounded subset $A$ of $X$ is weakly compact if it is merely compact in the topology $\sigma_E$, where $E=\ex B_{X^{*}}$. In \cite{godefroy1} Godefroy asked whether the same statement holds for an arbitrary boundary $B$, a question which has become known as the Boundary Problem. \newpage

Long since only partial positive answers were known, for example if $X=C(K)$ for some compact Hausdorff space $K$ (cf. \cite[Proposition 3]{cascales1}) or $X=\ell^1(I)$ for some set $I$ (cf. \cite[Theorem 4.9]{cascales2}). In \cite[Theorem 1.1]{spurny} the positive answer for $L_1$-preduals is contained. Moreover, the answer is positive if the set $A$ is additionally assumed to be convex (cf. \cite[p.44]{godefroy2}). It was only in 2008 that the positive answer to the Boundary Problem was found in full generality by Pfitzner in \cite{pfitzner}.\par 

An important tool in the study of the Boundary Problem is the so called Simons' equality:

\begin{theorem}[Simons, cf. \cite{simons2}] \label{th:SLS}
If $B$ is a boundary for $X$, then 
\begin{equation}
\sup_{x^{*}\in B}\limsup_{n\to \infty}x^{*}(x_n)=\sup_{x^{*}\in B_{X^{*}}}\limsup_{n\to \infty}x^{*}(x_n) \label{eq:SLS}
\end{equation}
holds for every bounded sequence $(x_n)_{n\in \mathbb{N}}$ in $X$.
\end{theorem}

\noindent In particular, it follows from Theorem 1.1 that the well-known Rainwater's theorem for the extreme points of the dual unit ball (cf. \cite{rainwater}) holds true for an arbitrary boundary:

\begin{corollary}[Simons, cf. \cite{simons1} or \cite{simons2}] \label{cor:RS conv thm}
If $B$ is a boundary for $X$, then a bounded sequence $(x_n)_{n\in \mathbb{N}}$ in $X$ is weakly convergent to $x\in X$ iff it is $\sigma_B$-convergent to $x$.
\end{corollary}

Pfitzner's proof also uses Simons' equality, as well as a quantitative version of Rosenthal's $\ell^1$-theorem due to Behrends (cf. \cite{behrends}) and an ingenious variant of Hagler-Johnson's construction. \par 

Next we recall the following known characterization of weak compactness (compare \cite[p.145-149]{holmes}, \cite[Theorem 5.5 and Exercise 5.19]{floret} as well as the proof of \cite[Theorem V.6.2]{dunford}). It is a strengthening of the usual Eberlein-\v{S}mulian theorem.

\begin{theorem} \label{th:convex hull char wcomp}
 Let $A$ be a bounded subset of $X$. Then the following assertions are equivalent:
\begin{enumerate}[\upshape(i)]
\item $A$ is weakly relatively compact.
\item For every sequence $(x_n)_{n\in \mathbb{N}}$ in A we have that 
\begin{equation*}
\bigcap _{k=1}^{\infty}\ncco{\set*{x_n:n\geq k}} \neq \emptyset.
\end{equation*}
\item For every sequence $(x_n)_{n\in \mathbb{N}}$ in $A$ there is some $x\in X$ such that 
\begin{equation*}
x^{*}(x)\leq \limsup_{n\to \infty}x^{*}(x_n) \ \ \forall x^{*}\in X^{*}.
\end{equation*}
\end{enumerate}
\end{theorem}

In \cite{moors} Moors proved a statement stronger than the equivalence of (i) and (ii), which also sharpens the result from \cite{bourgain}:

\begin{theorem}[Moors, cf. \cite{moors}] \label{th:moors char wc}
A bounded subset $A$ of $X$ is weakly relatively compact iff for every sequence $(x_n)_{n\in \mathbb{N}}$ in $A$ we have that
\begin{equation*}
\bigcap _{k=1}^{\infty} \cco{\set*{x_n:n\geq k}}{\sigma_E} \neq \emptyset,
\end{equation*}
where $E=\ex B_{X^{*}}$. In particular, $A$ is weakly relatively compact if it is merely relatively countably compact in the topology $\sigma_E$.
\end{theorem}

\noindent In fact, Moors gets this theorem as a corollary to the following one:

\begin{theorem}[Moors, cf. \cite{moors}] \label{th:moors}
 Let $A$ be an infinite bounded subset of $X$. Then there exists a countably infinite set $F\subseteq A$ with $\cco{F}{\sigma_E}=\ncco{F}$, where $E=\ex B_{X^{*}}$. In particular, for each bounded sequence $(x_n)_{n\in \mathbb{N}}$ in $X$ there is a subsequence $(x_{n_k})_{k\in \mathbb{N}}$ with $\cco{\set*{x_{n_k}:k\in \mathbb{N}}}{\sigma_E}\subseteq \ncco{\set*{x_n:n\in \mathbb{N}}}$.
\end{theorem}

The object of this paper is to give some results related to Theorem \ref{th:moors char wc} in the more general context of boundaries. In particular, we shall see, by a very slight modification of the construction from \cite{pfitzner}, that a `non-relative' version of \ref{th:moors char wc} holds for any boundary $B$ of $X$, see Theorem \ref{th:non-rel version}. \par 

Since we will also deal with some quantitative versions of Theorem \ref{th:moors char wc}, it is necessary to introduce a bit more of terminology, which stems from \cite{fabian}: Given $\varepsilon \geq 0$, a bounded subset $A$ of $X$ is said to be $\varepsilon$-weakly relatively compact (in short $\varepsilon$-WRC) provided that $\dist{x^{**}}{X}\leq \varepsilon$ for every element $x^{**}\in \cl{A}{w^{*}}$, where $w^{*}$ refers to the weak*-topology of $X^{**}$. For $\varepsilon=0$ this is equivalent to the classical case of weak relative compactness. \par 

\noindent The authors of \cite{fabian} used this notion to give a quantitative version of the well known theorem of Krein (cf. \cite[Theorem 2]{fabian}). In their proof they made use of double limit techniques in the spirit of Grothendieck. More precisely, they worked with the following definition: Let bounded subsets $A$ of $X$, $M$ of $X^{*}$ and $\varepsilon \geq 0$ be given. Then $A$ is said to have $\varepsilon$-interchangeable double limits with $M$ if for any two sequences $(x_n)_{n\in \mathbb{N}}$ in $A$ and $(x_m^{*})_{m\in \mathbb{N}}$ in $M$ we have 
\begin{equation*}
\abs*{\lim_{n\to \infty} \lim_{m\to \infty} x_m^{*}(x_n)-\lim_{m\to \infty} \lim_{n\to \infty} x_m^{*}(x_n)}\leq \varepsilon,
\end{equation*}
provided that all the limits involved exist. In this case we write $A\S \varepsilon \S M$. \par 

\noindent The connection to $\varepsilon$-WRC sets is given by the following proposition:

\begin{proposition}[Fabian et al., cf. \cite{fabian}] \label{prop:eps-WRC}
Let $A\subseteq X$ be bounded and $\varepsilon \geq 0$. Then the following hold:
\begin{enumerate}[\upshape(i)] 
\item If $A$ is $\varepsilon$-WRC, then $A\S 2\varepsilon \S B_{X^{*}}$.
\item If $A\S \varepsilon \S B_{X^{*}}$, then $A$ is $\varepsilon$-WRC.
\end{enumerate}
\end{proposition}

\noindent In case $\varepsilon=0$ this is the classical Grothendieck double limit criterion. For various other quantitative results on weak compactness we refer the interested reader to \cite{angosto}, \cite{cascales5}, \cite{cascales3} and \cite{fabian}. For some related results on weak sequential completeness, see also \cite{kalenda}. \par 

We are now ready to formulate and prove our results. However, it should be added that all of them can easily be derived from already known results and techniques.

\section{Results and proofs}
We begin with a quantitative version of Theorem \ref{th:convex hull char wcomp}. First we prove an easy lemma that generalizes the equivalence of (ii) and (iii) in said theorem (the proof is practically the same).

\begin{lemma} \label{lemma:quan convex hull char}
Let $B$ be a subset of $B_{X^{*}}$ that separates the points of $X$ and let $(x_n)_{n\in \mathbb{N}}$ be a sequence in $X$ as well as $x\in X$ and $\varepsilon\geq 0$. Then the following assertions are equivalent:
\begin{enumerate}[\upshape(i)]
\item $x\in \bigcap_{k=1}^{\infty}\cco{\paren*{\set*{x_n:n\geq k}+\varepsilon B_X}}{\sigma_B}$.
\item $x^{*}(x)\leq \limsup_{n\to \infty}x^{*}(x_n)+\varepsilon \ \ \forall x^{*}\in B_{X^{*}}\cap \lin{B}$.
\end{enumerate}
\end{lemma}

\begin{proof}
First we assume (i). It then directly follows that
\begin{equation*}
x^{*}(x)\in \ncco{\paren*{\set*{x^{*}(x_n):n\geq k}+[-\varepsilon,\varepsilon]}} \ \forall k\in \mathbb{N} \ \forall x^{*}\in B_{X^{*}}\cap \lin{B}.
\end{equation*}
Thus we also have $x^{*}(x)\leq \sup_{n\geq k} x^{*}(x_n)+\varepsilon$ for all $k\in \mathbb{N}$ and all $x^{*}\in B_{X^{*}}\cap \lin{B}$ and the assertion (ii) follows. \par 
Now we assume that (ii) holds and take $k\in \mathbb{N}$ arbitrary. Suppose that
\begin{equation*}
 x\not\in \cco{\paren*{\set*{x_n:n\geq k}+\varepsilon B_X}}{\sigma_B}.
\end{equation*}
 Then by the separation theorem we could find a functional $x^{*}\in (X,{\sigma_B})^{\prime}=\lin{B}$ with $\norm*{x^{*}}=1$ and a number $\alpha\in \mathbb{R}$ such that
\begin{equation*}
x^{*}(y)\leq \alpha < x^{*}(x) \ \ \forall y\in \cco{\paren*{\set*{x_n:n\geq k}+\varepsilon B_X}}{\sigma_B}.
\end{equation*}
It follows that 
\begin{equation*}
\limsup_{n\to \infty}x^{*}(x_n)+\varepsilon\leq \alpha < x^{*}(x),
\end{equation*}
a contradiction which ends the proof.
\end{proof}

Now we can give a quantitative version of the first equivalence in Theorem \ref{th:convex hull char wcomp}.

\begin{theorem} \label{th:quan convex hull char wcomp}
 Let $A\subseteq X$ be bounded and $\varepsilon \geq 0$. If for each sequence $(x_n)_{n\in \mathbb{N}}$ in $A$ we have
\begin{equation}
\bigcap_{k=1}^{\infty} \ncco{\paren*{\set*{x_n:n\geq k}+\varepsilon B_X}} \neq \emptyset, \label{eq:*}
\end{equation}
then $A$ is $2\varepsilon$-WRC. \par 
\noindent If $A$ is $\varepsilon$-WRC, then 
\begin{equation}
\bigcap_{k=1}^{\infty} \ncco{\paren*{\set*{x_n:n\geq k}+r B_X}} \neq \emptyset \label{eq:**}
\end{equation}
holds for every sequence $(x_n)_{n\in \mathbb{N}}$ in $A$ and every $r>\varepsilon$.
\end{theorem}

\begin{proof}
First we assume that \eqref{eq:*} holds for every sequence in $A$. Let $(x_n)_{n\in \mathbb{N}}$ and $(x_m^{*})_{m\in \mathbb{N}}$ be sequences in $A$ and $B_{X^{*}}$, respectively, such that the limits 
\begin{equation*}
\lim_{n\to \infty}\lim_{m\to \infty} x_m^{*}(x_n) \ \text{and} \ \lim_{m\to \infty}\lim_{n\to \infty} x_m^{*}(x_n)
\end{equation*}
exist. By assumption, we can pick an element
\begin{equation*}
x\in \bigcap_{k=1}^{\infty} \ncco{\paren*{\set*{x_n:n\geq k}+\varepsilon B_X}}.
\end{equation*}
From Lemma \ref{lemma:quan convex hull char} we conclude that 
\begin{equation}
\liminf_{n\to \infty}x^{*}(x_n)-\varepsilon \leq x^{*}(x) \leq \limsup_{n\to \infty}x^{*}(x_n)+\varepsilon \ \ \forall x^{*}\in B_{X^{*}}. \label{eq:1}
\end{equation}
It follows that 
\begin{equation}
\abs*{x_m^{*}(x)-\lim_{n\to \infty}x_m^{*}(x_n)}\leq \varepsilon \ \ \forall m\in \mathbb{N}. \label{eq:2}
\end{equation}
Now take a weak*-cluster point $x^{*}\in B_{X^{*}}$ of the sequence $(x_m^{*})_{m\in \mathbb{N}}$. Then 
\begin{equation}
\lim_{n\to \infty}\lim_{m\to \infty}x_m^{*}(x_n)=\lim_{n\to \infty}x^{*}(x_n). \label{eq:3}
\end{equation}
By \eqref{eq:1} we have 
\begin{equation}
\abs*{x^{*}(x)-\lim_{n\to \infty}x^{*}(x_n)}\leq \varepsilon. \label{eq:4}
\end{equation}
Since $x^{*}(x)-\lim_{m\to \infty}\lim_{n\to \infty}x_m^{*}(x_n)$ is a cluster point of the sequence $(x_m^{*}(x)-\lim_{n\to \infty}x_m^{*}(x_n))_{m\in \mathbb{N}}$ it follows from \eqref{eq:2} that
\begin{equation}
\abs*{x^{*}(x)-\lim_{m\to \infty}\lim_{n\to \infty}x_m^{*}(x_n)} \leq \varepsilon. \label{eq:5}
\end{equation}
From \eqref{eq:3}, \eqref{eq:4} and \eqref{eq:5} we get 
\begin{equation*}
\abs*{\lim_{m\to \infty}\lim_{n\to \infty}x_m^{*}(x_n)-\lim_{n\to \infty}\lim_{m\to \infty}x_m^{*}(x_n)}\leq 2\varepsilon.
\end{equation*}
Thus we have proved $A\S 2\varepsilon \S B_{X^{*}}$. Hence, by Proposition \ref{prop:eps-WRC}, $A$ is $2\varepsilon$-WRC. \par 
Now assume that $A$ is $\varepsilon$-WRC and take any sequence $(x_n)_{n\in \mathbb{N}}$ in $A$ as well as $r> \varepsilon$. Let $x^{**}\in \cl{A}{w^{*}}$ be a weak*-cluster point of $(x_n)_{n\in \mathbb{N}}$. Since $A$ is $\varepsilon$-WRC there is some $x\in X$ such that $\norm*{x^{**}-x}\leq r$. \par 
For every $x^{*}\in B_{X^{*}}$ the number $x^{**}(x^{*})$ is a cluster point of the sequence $(x^{*}(x_n))_{n\in \mathbb{N}}$ and thus 
\begin{equation*}
x^{*}(x)\leq \norm*{x-x^{**}}\norm*{x^{*}}+x^{**}(x^{*})\leq r+\limsup_{n\to \infty}x^{*}(x_n).
\end{equation*}
Lemma \ref{lemma:quan convex hull char} now yields
\begin{equation*}
x\in \bigcap_{k=1}^{\infty}\ncco{\paren*{\set*{x_n:n\geq k}+r B_X}}
\end{equation*}
and the proof is finished.
\end{proof}

As an immediate corollary we get

\begin{corollary} \label{cor:quan convex hull char wcomp}
If $A\subseteq X$ is bounded and $\varepsilon\geq 0$ such that
\begin{equation*}
\bigcap_{k=1}^{\infty}\paren*{\ncco{\set*{x_n:n\geq k}}+\varepsilon B_X} \neq \emptyset
\end{equation*}
for every sequence $(x_n)_{n\in \mathbb{N}}$ in $A$, then $A$ is $2\varepsilon$-WRC.
\end{corollary}

Now we can also prove a quantitative version of Theorem \ref{th:moors char wc}:

\begin{corollary} \label{cor:quan version moors char wc}
Let $A\subseteq X$ be bounded, $\varepsilon\geq 0$ and $E=\ex B_{X^{*}}$. If for each sequence $(x_n)_{n\in \mathbb{N}}$ in $A$ we have that 
\begin{equation*}
\bigcap_{k=1}^{\infty} \paren*{\cco{\set*{x_n:n\geq k}}{\sigma_E}+\varepsilon B_X} \neq \emptyset,
\end{equation*}
then $A$ is $2\varepsilon$-WRC.
\end{corollary}

\begin{proof}
Let $(x_n)_{n\in \mathbb{N}}$ be a sequence in $A$. By means of Theorem \ref{th:moors} and an easy diagonal argument we can find a subsequence $(x_{n_k})_{k\in \mathbb{N}}$ such that $\cco{\set*{x_{n_k}:k\geq l}}{\sigma_E}\subseteq \ncco{\set*{x_n:n\geq l}}$ for all $l$ (compare \cite[Corollary 0.2]{moors}). It then follows from our assumption that 
\begin{equation*}
\bigcap_{l=1}^{\infty} \paren*{\ncco{\set*{x_n:n\geq l}+\varepsilon B_X}} \neq \emptyset.
\end{equation*}
Hence, by Corollary \ref{cor:quan convex hull char wcomp}, $A$ is $2\varepsilon$-WRC.
\end{proof}

Next we observe that Moors' Theorem \ref{th:moors} does not only work for the extreme points of $B_{X^{*}}$ but also for any weak*-separable boundary.

\begin{theorem} \label{th:moors wssep}
Let $B$ be a weak*-separable boundary for $X$ and $A$ a bounded infinite subset of $X$. Then there is a countably infinite set $F\subseteq A$ such that $\ncco{F}=\cco{F}{\sigma_B}$. In particular, for every bounded sequence $(x_n)_{n\in \mathbb{N}}$ in $X$ there exists a subsequence $(x_{n_k})_{k\in \mathbb{N}}$ with $\cco{\set*{x_{n_k}:k\in \mathbb{N}}}{\sigma_B}\subseteq \ncco{\set*{x_n:n\in \mathbb{N}}}$.
\end{theorem}

\begin{proof}
The proof is completely analogous to that of Theorem \ref{th:moors} given in \cite{moors}, in fact it is even simpler, so we shall only sketch it. Arguing by contradiction, we suppose that for each countably infinite subset $F$ of $A$ there is an element $z\in \cco{F}{\sigma_B}\setminus \ncco{F}$. \par 
 Then we can show exactly as in \cite{moors} (using the Bishop-Phelps theorem (cf. \cite[Theorem 5.5]{fonf}) and the Hahn-Banach separation theorem) that for every sequence $(x_n)_{n\in \mathbb{N}}$ in $A$ for which the set $\set*{x_n:n\in \mathbb{N}}$ is infinite, there is an element
\begin{equation}
x\in \bigcap_{k=1}^{\infty} \cco{\set*{x_n:n\geq k}}{\sigma_B} \setminus \ncco{\set*{x_n:n\in \mathbb{N}}}. \label{eq:7}
\end{equation} \par 
\noindent We remark that the weak*-separability of $B$ is not needed for this step. \par 
Next we fix a sequence $(x_n)_{n\in \mathbb{N}}$ in $A$ whose members are distinct and a countable weak*-dense subset $\set*{x_m^{*}:m\in \mathbb{N}}$ of $B$. By the usual diagonal argument we may select a subsequence (again denoted by $(x_n)_{n\in \mathbb{N}}$) such that $\lim_{n\to \infty} x_m^{*}(x_n)$ exists for all $m$. \par 
We then choose an element $x$ according to \eqref{eq:7} and conclude that for each $m\in \mathbb{N}$ we have $\lim_{n\to \infty}x_m^{*}(x_n)=x_m^{*}(x)$. \par 
Now let $x^{*}\in B$ be arbitrary. Again as in \cite{moors} we will show that $\lim_{n\to \infty}x^{*}(x_n)=x^{*}(x)$. Suppose that this is not the case. Then there is an $\varepsilon >0$ such that $\abs*{x^{*}(x)-x^{*}(x_n)}>\varepsilon$ for infinitely many $n\in \mathbb{N}$. Let us assume $x^{*}(x_n)>\varepsilon+x^{*}(x)$ for infinitely many $n$ and arrange these indices in an increasing sequence $(n_k)_{k\in \mathbb{N}}$. By \eqref{eq:7} we can find 
\begin{equation*}
z\in \bigcap_{l=1}^{\infty} \cco{\set*{x_{n_k}:k\geq l}}{\sigma_B}\setminus \ncco{\set*{x_{n_k}:k\in \mathbb{N}}}.
\end{equation*}
It follows that $x_m^{*}(z)=\lim_{k\to \infty}x_m^{*}(x_{n_k})=x_m^{*}(x)$ for all $m$ and since $\set*{x_m^{*}:m\in \mathbb{N}}$ is weak*-dense in $B$ this implies $x^{*}(x)=x^{*}(z)$, whereas on the other hand $x^{*}(z)\geq \varepsilon+x^{*}(x)$, a contradiction. \par Thus $(x_n)_{n\in \mathbb{N}}$ is $\sigma_B$-convergent to $x$ and hence, by Corollary \ref{cor:RS conv thm} it is also weakly convergent to $x$, which in turn implies $x\in \ncco{\set*{x_n:n\in \mathbb{N}}}$, contradicting the choice of $x$.
\end{proof}

 Note that the assumption of weak*-separability of $B$ is fulfilled, in particular, if $X$ is separable, for then the weak*-topology on $B_{X^{*}}$ is metrizable. As an immediate corollary we get \ref{cor:quan version moors char wc} for weak*-separable boundaries.

\begin{corollary} \label{cor:moors quan wssep}
Let $B$ be a boundary for $X$ and $A$ a bounded subset of $X$ as well as $\varepsilon \geq 0$. If $B$ is weak*-separable (in particular, if $X$ is separable) and for each sequence $(x_n)_{n\in \mathbb{N}}$ in $A$ we have
\begin{equation*}
\bigcap_{k=1}^{\infty} \paren*{\cco{\set*{x_n:n\geq k}}{\sigma_B}+\varepsilon B_X} \neq \emptyset,
\end{equation*}
then $A$ is $2\varepsilon$-WRC.
\end{corollary}

\begin{proof}
Exactly as the proof of Corollary \ref{cor:quan version moors char wc}.
\end{proof}

Let us now consider Banach spaces of a certain type, namely the case $X=C(K)$ for some compact Hausdorff space $K$ or $X=\ell^1(I)$ for some index set $I$. In \cite{cascales1} respectively \cite{cascales2} Cascales et al. found the positive solution to the Boundary Problem for these types of spaces. In fact, they even proved a stronger statement, namely that in the above cases the space $(X,\sigma_B)$ is angelic\footnote{See \cite{cascales1} or \cite{floret} for the definition and background.} for every boundary $B$ of $X$. In order to get the statement of Corollary \ref{cor:moors quan wssep} for arbitrary boundaries in $C(K)$- and $\ell^1(I)$-spaces we shall need the following easy lemma.

\begin{lemma} \label{lemma:unifying lemma}
Let $T$ and $S$ be subsets of $X^{*}$ such that for every countable set $D\subseteq X$ and every $x^{*}\in T$ there is some $y^{*}\in S$ such that $x^{*}(x)=y^{*}(x)$ for all $x\in D$. \par \noindent
Then for every countable set $D\subseteq X$ we have $\cco{D}{\sigma_S} \subseteq \cco{D}{\sigma_T}$.
\end{lemma}

\begin{proof}
Let $D\subseteq X$ be countable and take any $x\in \cco{D}{\sigma_S}$. Further, fix $x_1^{*},\dots ,x_{n}^{*}\in T$ and $\varepsilon> 0$. By assumption we can find $y_{1}^{*},\dots ,y_{n}^{*}\in S$ such that
\begin{equation*}
x_{i}^{*}(y)=y_{i}^{*}(y) \ \  \forall y\in D\cup \set*{x} \ \  \forall i=1,\dots ,n.
\end{equation*}
But then the same equality holds for every $y\in \co{\paren*{D\cup \set*{x}}}$ and since $x\in \cco{D}{\sigma_S}$ we may select some $y\in \co{D}$ with $\abs*{y_{i}^{*}(x)-y_{i}^{*}(y)} \leq \varepsilon$ for all $i=1,\dots ,n$. It follows that $\abs*{x_{i}^{*}(x)-x_{i}^{*}(y)} \leq \varepsilon$ for $i=1,\dots ,n$ and the proof is finished.
\end{proof}

According to the results of Cascales et al. the condition of Lemma \ref{lemma:unifying lemma} is fulfilled if $X=C(K)$ or $X=\ell^1(I)$, $T=\ex{B_{X^{*}}}$ and $S$ is any boundary for $X$ (see \cite[Lemma 1]{cascales1} for $X=C(K)$ and the proof of \cite[Theorem 4.9]{cascales2} for $X=\ell^1(I)$), thus we immediately get the following lemma.

\begin{lemma} \label{lemma:CK-L1}
If $X=C(K)$ for some compact Hausdorff space $K$ or $X=\ell^1(I)$ for some index set $I$ and $B$ is any boundary for $X$, then for every countable set $D\subseteq X$ we have $\cco{D}{\sigma_B}\subseteq \cco{D}{\sigma_E}$, where $E=\ex{B_{X^{*}}}$.
\end{lemma}

From Lemma \ref{lemma:CK-L1} and Corollary \ref{cor:quan version moors char wc} we now get the desired result.

\begin{corollary} \label{cor:CK-L1}
If $X=C(K)$ for some compact Hausdorff space $K$ or $X=\ell^1(I)$ for some set $I$ and $B$ is any boundary for $X$ as well as $A\subseteq X$ a bounded set and $\varepsilon \geq 0$ such that for every sequence $(x_n)_{n\in \mathbb{N}}$ in $A$ we have 
\begin{equation*}
\bigcap_{k=1}^{\infty}\paren*{\cco{\set*{x_n:n\geq k}}{\sigma_B}+\varepsilon B_{X}} \neq \emptyset,
\end{equation*}
then $A$ is $2\varepsilon$-WRC.
\end{corollary}

Next we turn to spaces not containing isomorphic copies of $\ell^1$. It is known that for such spaces one has $\cco{B}{\gamma}=B_{X^{*}}$ for every boundary $B$ of $X$, where we denote by $\gamma$ the topology on $X^{*}$ of uniform convergence on bounded countable subsets of $X$ (cf. \cite[Theorem 5.4]{cascales4}). \par 

We will also need two easy lemmas.

\begin{lemma} \label{lemma:gamma1}
Let $A\subseteq X$ and $S\subseteq X^{*}$ be bounded as well as $\varepsilon \geq 0$ such that $A\S \varepsilon \S S$. Then we also have $A\S \varepsilon \S \cl{S}{\gamma}$.
\end{lemma}

\begin{proof}
Let $(x_n)_{n\in \mathbb{N}}$ and $(x_m^{*})_{m\in \mathbb{N}}$ be sequences in $A$ and $\cl{S}{\gamma}$, respectively, such that the limits 
\begin{equation*}
\lim_{n\to \infty}\lim_{m\to \infty} x_m^{*}(x_n) \ \text{and} \ \lim_{m\to \infty}\lim_{n\to \infty} x_m^{*}(x_n)
\end{equation*}
exist. For each $m\in \mathbb{N}$ we can pick a functional $\tilde{x}_m^{*}\in S$ with
\begin{equation*}
\abs*{x_m^{*}(x_n)-\tilde{x}_m^{*}(x_n)}\leq \frac{1}{m} \ \ \forall n\in \mathbb{N}.
\end{equation*}
By the usual diagonal argument, choose a subsequence $(x_{n_k})_{k\in \mathbb{N}}$ such that $\lim_{k\to \infty}\tilde{x}_m^{*}(x_{n_k})$ exists for all $m$. It then easily follows that 
\begin{align*}
&\lim_{m\to \infty}\lim_{n\to \infty} x_m^{*}(x_n)=\lim_{m\to \infty}\lim_{k\to \infty}\tilde{x}_m^{*}(x_{n_k}) \ \text{and} \ \\
&\lim_{n\to \infty}\lim_{m\to \infty}x_m^{*}(x_n)=\lim_{k\to \infty}\lim_{m\to \infty}\tilde{x}_m^{*}(x_{n_k}).
\end{align*}
Since $A\S \varepsilon \S S$, we conclude that
\begin{equation*}
\abs*{\lim_{n\to \infty}\lim_{m\to \infty}x_m^{*}(x_n)-\lim_{m\to \infty}\lim_{n\to \infty}x_m^{*}(x_n)}\leq \varepsilon,
\end{equation*}
finishing the proof.
\end{proof}

\begin{lemma} \label{lemma:gamma2}
Let $(x_n)_{n\in \mathbb{N}}$ be a bounded sequence in $X$ and $B\subseteq B_{X^{*}}$ such $\cco{B}{\gamma}=B_{X^{*}}$. Then 
\begin{equation*}
\cco{\set*{x_n:n\in \mathbb{N}}}{\sigma_B}=\ncco{\set*{x_n:n\in \mathbb{N}}}.
\end{equation*}
\end{lemma}

\begin{proof}
Take $x\in \cco{\set*{x_n:n\in \mathbb{N}}}{\sigma_B}$ and let $\varepsilon > 0$ and $x_1^{*},\dots , x_k^{*}\in B_{X^{*}}$ be arbitrary. By assumption, we can find $\tilde{x}_1^{*},\dots ,\tilde{x}_k^{*}\in \co{B}$ such that for $i=1,\dots ,k$ we have
\begin{equation*}
\abs*{\tilde{x}_i^{*}(x_n)-x_i^{*}(x_n)}\leq \varepsilon \ \forall n\in \mathbb{N} \  \text{and} \ \abs*{\tilde{x}_i^{*}(x)-x_i^{*}(x)}\leq \varepsilon.
\end{equation*}
It follows that 
\begin{equation*}
\abs*{\tilde{x}_i^{*}(y)-x_i^{*}(y)}\leq \varepsilon \ \ \forall y\in \co{\paren*{\set*{x_n:n\in \mathbb{N}}\cup \set*{x}}} \ \forall i=1,\dots ,k.
\end{equation*}
Now take some element $y\in \co{\set*{x_n:n\in \mathbb{N}}}$ with $\abs*{\tilde{x}_i^{*}(y)-\tilde{x}_i^{*}(x)}\leq \varepsilon$ for all $i=1,\dots ,k$. \par 
\noindent Employing the triangle inequality we can deduce $\abs*{x_i^{*}(x)-x_i^{*}(y)}\leq 3\varepsilon$, which ends the proof.
\end{proof}

As an immediate consequence of Lemma \ref{lemma:gamma2}, Corollary \ref{cor:quan convex hull char wcomp} and the aforementioned result \cite[Theorem 5.4]{cascales4} we get the following corollary.

\begin{corollary} \label{cor:l1 not in X}
Suppose $\ell^1\not\subseteq X$ and let $B$ be a boundary for $X$. If $A\subseteq X$ is bounded and $\varepsilon \geq 0$ such that for each sequence $(x_n)_{n\in \mathbb{N}}$ in $A$ we have 
\begin{equation*}
\bigcap_{k=1}^{\infty} \paren*{\cco{\set*{x_n:n\geq k}}{\sigma_B}+\varepsilon B_X} \neq \emptyset,
\end{equation*}
then $A$ is $2\varepsilon$-WRC.
\end{corollary}

We can further get a kind of `boundary double limit criterion'.

\begin{proposition} \label{prop:boundary-double-limit-crit}
 Let $B$ be a boundary for $X$ as well as $\varepsilon \geq 0$ and $A\subseteq X$ be bounded such that $A\S \varepsilon \S B$. Then $A$ is $2\varepsilon$-WRC. If $\ell^1\not\subseteq X$, then $A$ is even $\varepsilon$-WRC.
\end{proposition}

\begin{proof}
From \cite[Theorem 3.3]{cascales3} it follows that we also have $A\S \varepsilon \S \co{B}$. Since $B$ is a boundary for $X$ the Hahn-Banach separation theorem implies $B_{X^{*}}=\cco{B}{w^{*}}$. Therefore it follows from \cite[Lemma 3]{angosto} that $A\S 2\varepsilon \S B_{X^{*}}$. Thus by (ii) of Proposition \ref{prop:eps-WRC} $A$ is $2\varepsilon$-WRC.\footnote{This proof also works under the weaker assumption that $B$ is only norming for $X$, i.e. $\norm*{x}=\sup_{b\in B}b(x)$ for all $x\in X$, because in this case we also have $B_{X^{*}}=\cco{B}{w^{*}}$ by the separation theorem.} \par 
 Moreover, if $\ell^1\not\subseteq X$ then we even have $B_{X^{*}}=\cco{B}{\gamma}$ by the already cited \cite[Theorem 5.4]{cascales4}. Hence $A\S \varepsilon \S B_{X^{*}}$ by Lemma \ref{lemma:gamma1}, thus $A$ is $\varepsilon$-WRC.
\end{proof}

Our final aim in this note is to prove a `non-relative' version of Theorem \ref{th:moors char wc} for arbitrary boundaries. To do so, we will use the techniques of Pfitzner from \cite{pfitzner}. More precisely, we can get the following slight generalization of the ``in particular case'' of \cite[Proposition 8]{pfitzner}. Recall that an $\ell^1$-sequence in $X$ is simply a sequence equivalent to the canonical basis of $\ell^1$.

\begin{proposition} \label{prop:non-rel version l1}
 Let $B$ be a boundary for $X$. If  $A\subseteq X$ is bounded and for every sequence $(x_n)_{n\in \mathbb{N}}$ in $A$ we have 
\begin{equation}
A\cap \bigcap _{k=1}^{\infty} \cco{\set*{x_n:n\geq k}}{\sigma_B} \neq \emptyset,
\end{equation}
then $A$ does not contain an $\ell^1$-sequence.
\end{proposition}

\begin{proof}
The proof is completely analogous to that of \cite[Proposition 8]{pfitzner}, therefore we will only give a very brief sketch. We use the notation and definitions from \cite{pfitzner}. Arguing by contradiction, we assume that there is an $\ell^1$-sequence $(x_n)_{n\in \mathbb{N}}$ in $A$. By \cite[Lemma 2]{pfitzner} we may assume that $(x_n)_{n\in \mathbb{N}}$ is $\delta$-stable. We take a sequence $(\alpha_k)_{k\in \mathbb{N}}$ of positive numbers decreasing to zero. By \cite[Lemma 7]{pfitzner} we can find $\varepsilon\geq 1/2\tilde{\delta}_{B}(x_n)=1/2\tilde{\delta}(x_n)>0$, a sequence $(b_k)_{k\in \mathbb{N}}$ in $B$ and a tree $(\Omega_{\sigma})_{\sigma\in S}$ such that for each $k\in \mathbb{N}$ and every $\sigma,\sigma^{\prime}\in S_k$ with $\sigma_k=0$ and $\sigma_k^{\prime}=1$ we have
\begin{equation*}
b_k(x_n-x_{n^{\prime}})\geq 2\varepsilon (1-\alpha_k) \ \ \forall n\in \Omega_{\sigma},n^{\prime}\in \Omega_{\sigma^{\prime}}.
\end{equation*}
It follows that the same inequality holds for every $x\in \cco{\set*{x_n:n\in \Omega_{\sigma}}}{\sigma_B}$ and $x^{\prime}\in \cco{\set*{x_{n^{\prime}}:n^{\prime}\in \Omega_{\sigma^{\prime}}}}{\sigma_B}$. \par 
\noindent Now using our hypothesis we can proceed completely analogous to the proof of the claim in \cite[Proposition 8]{pfitzner} to find a sequence $(y_m)_{m\in \mathbb{N}}$ in $A\cap \bigcap_{k=1}^{\infty}\cco{\set*{x_n:n\geq k}}{\sigma_B}$ such that
\begin{equation*}
b_k(y_m-y_{m^{\prime}})\geq 2\varepsilon (1-\alpha_k) \ \ \forall m\leq k < m^{\prime}.
\end{equation*}
Next we take an element
\begin{equation*}
y\in A\cap \bigcap_{k=1}^{\infty}\cco{\set*{y_m:m\geq k}}{\sigma_B}.
\end{equation*}
As in the proof of \cite[Proposition 8]{pfitzner} we put
\begin{equation*}
x=\sum_{m=1}^{\infty} 2^{-m}(y_m-y)
\end{equation*}
and proceed again exactly as in the proof of \cite[Proposition 8]{pfitzner} to show that $\norm*{y_m-y}\leq 2\varepsilon$ for all $m$ and $\norm*{x}=2\varepsilon$. Finally, taking a functional $b\in B$ with $b(x)=\norm*{x}$ we obtain $b(y)=2\varepsilon+b(y)$ and with this contradiction the proof is finished.
\end{proof}

Now we can get the final result.

\begin{theorem} \label{th:non-rel version}
 Let $B$ be a boundary for $X$ and $A\subseteq X$ be bounded. Then the following assertions are equivalent:
\begin{enumerate}[\upshape(i)]
\item $A$ is countably compact in the topology $\sigma_B$.
\item For every sequence $(x_n)_{n\in \mathbb{N}}$ in $A$ we have 
\begin{equation*}
A\cap \bigcap_{k=1}^{\infty} \cco{\set*{x_n:n\geq k}}{\sigma_B} \neq \emptyset.
\end{equation*}
\item For every sequence $(x_n)_{n\in \mathbb{N}}$ in $A$ there is some $x\in A$ with
\begin{equation*}
x^{*}(x)\leq \limsup_{n\to \infty}x^{*}(x_n) \ \ \forall x^{*}\in \lin{B}.
\end{equation*}
\item $A$ is weakly compact.
\end{enumerate}
\end{theorem}

\begin{proof}
The implications (i) $\Rightarrow$ (ii) and (iv) $\Rightarrow$ (i) are clear and the equi\-valence of (ii) and (iii) follows from Lemma \ref{lemma:quan convex hull char}. It only remains to prove (ii) $\Rightarrow$ (iv). \par 
Let us assume that (ii) holds and take an arbitrary sequence $(x_n)_{n\in \mathbb{N}}$ in $A$. By Proposition \ref{prop:non-rel version l1} no subsequence of $(x_n)_{n\in \mathbb{N}}$ is an $\ell^1$-sequence and thus Rosenthal's theorem (cf. \cite{behrends} or \cite[Theorem 10.2.1]{albiac}) applies to yield a subsequence $(x_{n_k})_{k\in \mathbb{N}}$ which is weakly Cauchy. Now choose an element 
\begin{equation*}
x\in A\cap \bigcap_{l=1}^{\infty}\cco{\set*{x_{n_k}:k\geq l}}{\sigma_B}.
\end{equation*}
It easily follows that $\lim_{k\to \infty}b(x_{n_k})=b(x)$ for all $b\in B$. By Corollary \ref{cor:RS conv thm} $(x_{n_k})_{k\in \mathbb{N}}$ is weakly convergent to $x$. \par 
\noindent Thus we have shown that $A$ is weakly sequentially compact. Hence it is also weakly compact, by the Eberlein-\v{S}mulian theorem.
\end{proof}

\noindent {\em Remark 1.} It is proved in \cite[Remark 10]{fabian} that for $X=\ell^1$ the statement $B_X\S \varepsilon \S B_{X^{*}}$ is false for every $0<\varepsilon<2$. An alternative proof of this fact is given \cite[Example 5.2]{cascales5}. It is further proved in \cite[Remark 10]{fabian} that every separable Banach space $X$ which contains an isomorphic copy of $\ell^1$ can be equivalently renormed such that, in this renorming, the statement $B_X\S \varepsilon \S B_{X^{*}}$ is false for every $0<\varepsilon<2$. The proof makes use of the notion of octahedral norms. \par 

We wish to point out here that the argument from \cite[Example 5.2]{cascales5} can be carried over to arbitrary Banach spaces containing a copy of $\ell^1$, precisely we have the following proposition.

\begin{proposition}
 If $X$ is a (not necessarily separable) Banach space which contains $\ell^1$ than the statement $B_X\S \varepsilon \S B_{X^{*}}$ (in the original norm of X) is false for every $0<\varepsilon<2$.
\end{proposition}

\begin{proof}
Take $0<\varepsilon<2$ arbitrary and fix $0<\delta <1$ such that $2(1-\delta)>\varepsilon$. Since $X$ contains $\ell^1$ we may find, with the aid of James' $\ell^1$-distortion theorem (cf. \cite[Theorem 10.3.1]{albiac}), a sequence $(x_n)_{n\in \mathbb{N}}$ in the unit sphere of $X$ such that $T: \ell^1 \rightarrow X$ defined by 
\begin{equation*}
Ty=\sum_{k=1}^{\infty}\alpha_k x_k \ \ \forall y=(\alpha_n)_{n\in \mathbb{N}}\in \ell^1
\end{equation*}
is an isomorphism (onto $U=\operatorname{ran}T$) with $\norm*{T^{-1}}\leq (1-\delta)^{-1}$. Consequently, the adjoint $T^{*}:U^{*} \rightarrow \ell^{\infty}$ is as well an isomorphism with $\norm*{(T^{*})^{-1}}\leq (1-\delta)^{-1}$. Now we can define as in \cite[Example 5.2]{cascales5} for each $n\in \mathbb{N}$ a norm one functional $y_n^{*}\in \ell^{\infty}$ by 
\begin{equation*}
y_n^{*}(m)=
\begin{cases}
 1, &  \text{if\ } m\leq n \\
-1, & \text{if\ } m>n.
\end{cases}
\end{equation*}
Put $u_n^{*}=(T^{*})^{-1}y_n^{*}$  for all $n\in \mathbb{N}$. Then $\norm*{u_n^{*}}\leq(1-\delta)^{-1}$ and hence by the Hahn-Banach extension theorem we can find $x_n^{*}\in B_{X^{*}}$ with $x_n^{*}|_U=(1-\delta)u_n^{*}$ for all $n\in \mathbb{N}$. It follows that 
\begin{equation*}
\abs*{\lim_{n\to \infty}\lim_{m\to \infty} x_n^{*}(x_m)-\lim_{m\to \infty}\lim_{n\to \infty}x_n^{*}(x_m)}=2(1-\delta)>\varepsilon
\end{equation*}
and the proof is finished.
\end{proof}

 In the notation of \cite{cascales5} we have proved $\gamma(B_X)=2$ for every Banach space $X$ containing an isomorphic copy of $\ell^1$, which implies that the value of $B_X$ under all other measures of weak non-compactness considered in \cite{cascales5} is equal to one (again compare \cite[Example 5.2]{cascales5}). So in a certain sense a Banach space containing $\ell^1$ is `as non-reflexive as  possible'.
\par \ \\

\noindent {\em Remark 2.} Shortly after the first version of this paper was published on the web, the author received a message from Prof. Warren B. Moors, who kindly pointed out to him that the above Lemma \ref{lemma:CK-L1} probably also holds true if $X$ is an $L_1$-predual\footnote{Recall that a Banach space $X$ is called an $L_1$-predual if $X^{*}$ is isometric to $L_1(\mu)$ for some suitable measure $\mu$.} (which includes all $C(K)$-spaces), refering to the paper \cite{moors2}. Indeed, from \cite[Theorem 3]{moors2} one can easily get the following result: If $B$ is a boundary for the $L_1$-predual $X$ and $E=\ex B_{X^{*}}$, then
\begin{equation*}
\cco{\set*{x_n:n\in \mathbb{N}}}{\sigma_B}\subseteq \cco{\set*{x_n:n\in \mathbb{N}}}{\sigma_E}
\end{equation*}
holds for every sequence $(x_n)_{n\in \mathbb{N}}$ in $X$. For the proof just apply \cite[Theorem 3]{moors2} to the countable set $\operatorname{co}_{\mathbb{Q}}\set*{x_n:n\in \mathbb{N}}$ consisting of all convex combinations of the $x_n$'s with rational coefficients.\par
It follows that Corollary \ref{cor:CK-L1} also carries over to arbitrary boundaries of $L_1$-preduals. \par

\ \\{\sc Acknowledgement.} The author wishes to express his gratitude to Prof. Warren B. Moors for providing him with the important hint already mentioned in the remark above and to the anonymous referee for multiple comments and suggestions (in particular for proposing Lemma \ref{lemma:unifying lemma}) which improved the exposition of the results.

{\sc \noindent Department of Mathematics \\ Freie Universit\"at Berlin \\ Arnimallee 6, 14195 Berlin \\ Germany \\}
{\it E-mail address: \href{mailto:hardtke@math.fu-berlin.de}{\tt hardtke@math.fu-berlin.de}}

\end{document}